\documentclass[12pt,a4paper]{amsart}

\usepackage{amscd}
\usepackage{graphicx}
\usepackage{amssymb,amsfonts,latexsym}

\newif\iffurther
\furtherfalse

\input xy
\xyoption {all}

\newtheorem{thm}{Theorem}[section]
\newtheorem{cor}[thm]{Corollary}
\newtheorem{lem}[thm]{Lemma}
\newtheorem{prop}[thm]{Proposition}
\theoremstyle{definition}

\newtheorem{defn}[thm]{Definition}
\theoremstyle{remark}
\newtheorem{rem}[thm]{Remark}

\newtheorem{exam}[thm]{Example}
\numberwithin{equation}{section}

\newcommand\Tref[1]{{Theorem~\ref{#1}}}
\newcommand\Pref[1]{{Proposition~\ref{#1}}}
\newcommand\Lref[1]{{Lemma~\ref{#1}}}
\newcommand\Eref[1]{{Example~\ref{#1}}}
\newcommand\Cref[1]{{Corollary~\ref{#1}}}
\newcommand\Rref[1]{{Remark~\ref{#1}}}

\newcommand\Sref[1]{{Section~\ref{#1}}}

\newcommand\eq[1]{{(\ref{#1})}}

\long\def\forget#1\forgotten{}

\newcommand{\mQ}{\mathbb{Q}}
\newcommand{\mC}{\mathbb{C}}
\newcommand{\Z}{\mathbb{Z}}
\newcommand{\F}{\mathbb{F}}
\def\s{{\sigma}}

\newcommand{\ra}{\rightarrow}
\newcommand\sg[1]{{\left<{#1}\right>}}
\newcommand\subjectto{{\,|\ }}

\newcommand\tensor[1][]{{\otimes_{#1}}}
\DeclareMathOperator\Gal{Gal}
\def\({\left(}
\def\){\right)}
\def\divides{{\,|\,}}
\def\co{{\,{:}\,}}

\newcommand\card[1]{{\left|{#1}\right|}}
\newcommand\dimcol[2]{{[{#1}\!:\!{#2}]}}

\newcommand\oline[1] {{\overline{#1}}}

\newcommand\Br{{\operatorname{Br}}}

\newcommand\tame[1]{{{#1}_{\operatorname{tr}}}}

\newcommand\res[1][{}] {{\operatorname{res}_{#1}}}

\newcommand\inv{{\operatorname{inv}}}
\newcommand\Aut{{\operatorname{Aut}}}

\newcommand\sub{{\,\subset\,}}
\renewcommand\Im{{\operatorname{Im}}}

\newcommand\GL[1][n] {{\operatorname{GL}_{#1}}}

\newcommand\ab{{\operatorname{\,ab}}}

\def\Casei{{(\ref{casei})}}       
\def\Caseii{{(\ref{caseii})}}     
\def\Caseiii{{(\ref{caseiii})}}   
\def\Caseiv{{(\ref{caseiv})}}     
\def\Casev{{(\ref{casev})}}       
\def\Casevi{{(\ref{casevi})}}     
\def\Casevii{{(\ref{casevii})}}   
\def\Caseviii{{(\ref{caseviii})}} 
\def\Casefinal{{(\ref{casevii})}} 
\def\Caseix{{(\ref{caseix})}}     
\def\Casenor{{(\ref{casenor})}} 
\def\final{{\ref{casevii}}} 

\renewcommand\nu{{v}}
\renewcommand\varpi{{w}}

\begin{document}

\title
{Fields of definition for admissible groups}%

\def\Tech{Deptartment of Mathematics, Technion-Israel Institute of Technology, Haifa 32000, Israel}

\author{ Danny Neftin}
\address{\Tech}
\email{neftind@tx.technion.ac.il}%

\author{ Uzi Vishne }

\def\BIU{Deptartment of Mathematics, Bar Ilan University, Ramat Gan 52900, Israel}
\address{\BIU}
\email{vishne@math.biu.ac.il}

\begin{abstract}
A finite group $G$ is admissible over a field $M$ if there is a
division algebra whose center is $M$ with a maximal subfield
$G$-Galois over $M$. We consider nine possible notions of being
admissible over $M$ with respect to a subfield $K$ of $M$, where
the division algebra, the maximal subfield or the Galois group are
asserted to be defined over $K$. We completely determine the
logical implications between all variants.
\end{abstract}

\maketitle

\section{Introduction}

A group $G$ is admissible over a field $M$ if there
is a $G$-crossed product $M$-division algebra, namely a division
algebra $D$ whose center is $M$ with a maximal subfield $L$ which
is Galois over $M$ with Galois group $G$.

Given a subfield $K$ of $M$ and a group $G$ which is admissible
over $M$, one may ask how well can the admissibility be realized over
$K$. For example, $G$ can be already $K$-admissible, with a
$G$-crossed product over $K$ which remains a division algebra
after scalar extension to $M$. Failing this strong assumption, it
is still possible that $G$ is both $K$ and $M$-admissible; that
the $G$-crossed product $D$ is defined over $K$ (namely, $D = D_0
\tensor[K] M$ for a suitable division algebra over $K$); that $L$
is defined and Galois over $K$ (namely $L = L_0 \tensor[K] M$
where $L_0/K$ is $G$-Galois); or that $L$ is merely defined over
$K$.

This paper studies nine variations of $M$-admissibility of a group
$G$, with respect to a fixed subfield $K$ of $M$. We provide a
complete diagram of implications between those conditions (see
\Sref{sec:cond}). Furthermore, we provide counterexamples to every
implication which is not proved before with $G$ being a $p$-group
and $M$ a number field (see \Sref{sec4}).

It turns out that for $G$ cyclic and $M$ a number field, eight of
the nine conditions are satisfied (see \Sref{sec:cyclic}).
We shall also consider tame admissibility which is the type of admissibility that is best understood (see e.g. \cite{PaperI}) 
and show that these eight variants coincide with respect to tame admissibility.

The difference between tame and wild admissibility is an essential ingredient in the construction of counterexamples in \Sref{sec4}.


\section{Conditions on the field of definition}\label{sec:cond}

\subsection{The nine variations}
Let $K$ be a field and $G$ a finite group. We shall say that a field $L$ is
\emph{$M$-adequate} if it is a maximal subfield in some division
algebra whose center is $K$. We shall say that $L$ is a $G$-\emph{extension} of $K$ if
$L/K$ is a Galois extension with Galois group $\Gal(L/K)\cong G$.

Let $M/K$ a finite extension. One way to study the condition
\begin{enumerate}
\item \label{casei}
$G$ is $M$-admissible
\end{enumerate}
is by refining it to require that the
crossed-product division algebra or its maximal subfield are defined over $K$ (we say that a field or an
algebra over $M$ is defined over $K$ if it is obtained by
 scalar extension from $K$ to $M$).

Condition \eq{casei} requires
the existence of an $M$-adequate $G$-extension $L/M$. Three ways
in which this field can be related to $K$ provide the following
variants:
\begin{enumerate}
\setcounter{enumi}{1} 
\item \label{casenor}
there is an $M$-adequate $G$-extension $L/M$ for which $L$ is Galois over $K$;
\item \label{caseii}
there is an $M$-adequate $G$-extension $L/M$ for which $L$
is defined over $K$;
\item \label{caseiii}
there is an $M$-adequate $G$-extension $L/M$ so that
$L=L_0\otimes M$ and $\Gal(L_0/K) \cong G$.
\end{enumerate}
For the algebra $D$ to be defined over $K$, we may require that:
\begin{enumerate}
\setcounter{enumi}{4} 
\item \label{caseiv}
there is a $K$-division algebra $D_0$ and a $G$-extension
$L/M$ for which $L$ is a maximal subfield of $D_0\otimes M$;
\item \label{casev}
there is a $K$-division algebra $D_0$ and a maximal subfield $L_0$
which is a $G$-extension of $K$ so that $L_0\cap M=K$ and $L=L_0M$
is a maximal subfield of the division algebra $D=D_0\otimes M$.
\end{enumerate}
If $L = L_0 \tensor[K] M$, the interaction between $L_0$ and $L$
may involve the division algebras:
\begin{enumerate}
\setcounter{enumi}{6} 
\item \label{casevi}
there is a $K$-adequate $G$-extension $L_0/K$ for which $L_0M$ is
an $M$-adequate $G$-extension;
\item \label{caseviii}
there is a $K$-adequate $G$-extension $L_0/K$ for which $L_0M$ is
an $M$-adequate $G$-extension.
\end{enumerate}
And finally we have the double condition
\begin{enumerate}
\setcounter{enumi}{8}
\item \label{casevii}
$G$ is both $K$-admissible and $M$-admissible.

\end{enumerate}

We provide a diagrammatic description of each condition, for easy
reference. Inclusion is denoted by a vertical line, and diagonal
lines show the extension of scalars from $K$ to $M$. A vertical
line is decorated by $G$ if the field extension is $G$-Galois.
Note that in some cases (\Caseiii, \Casev \, and \Casevi) the fact
that the extension $L_0/K$ is Galois implies the same condition on
the extension $L/M$.
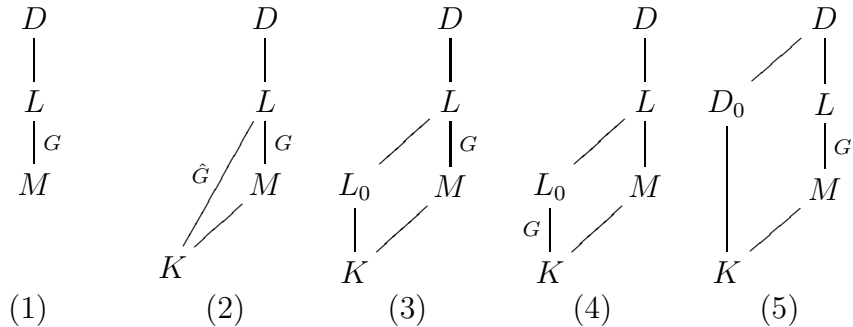
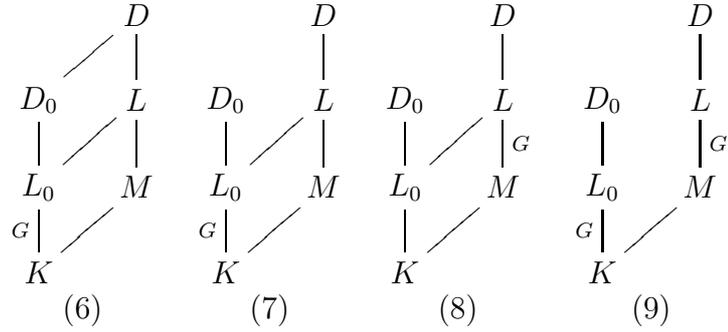
\begin{equation*} \begin{array}{lcccc}
                    \xymatrix@R=16pt@C=16pt{    D \ar@{-}[d] & \\
                               L \ar@{-}[d]^{G} & \\
                               M  }
                     & \xymatrix@R=16pt@C=16pt{ & D \ar@{-}[d]  \\
                              & L \ar@{-}[d]^{G}\ar@{-}[ddl]_{\hat{G}} \\
                              & M \ar@{-}[dl] \\
                              K  }
                     & \xymatrix@R=16pt@C=16pt{ & D \ar@{-}[d]  \\
                              & L \ar@{-}[dl]\ar@{-}[d]^{G} \\
                              L_0 \ar@{-}[d] & M \ar@{-}[dl] \\
                              K  }
                   & \xymatrix@R=16pt@C=16pt{ & D \ar@{-}[d]  \\
                              & L \ar@{-}[dl]\ar@{-}[d] \\
                              L_0 \ar@{-}[d]_{G} & M \ar@{-}[dl] \\
                              K  }
                   & \xymatrix@R=16pt@C=16pt{ & D \ar@{-}[dl]\ar@{-}[d]  \\
                              D_0 \ar@{-}[dd] & L \ar@{-}[d]^{G} \\
                               &  M \ar@{-}[dl] \\
                              K }
                  \\
                   \Casei & \Casenor & \Caseii & \Caseiii & \Caseiv
                  \end{array}
\end{equation*}

\begin{equation*} \begin{array}{cccc}
                    \xymatrix@R=16pt@C=16pt{ & D \ar@{-}[dl]\ar@{-}[d]  \\
                              D_0 \ar@{-}[d] & L \ar@{-}[dl]\ar@{-}[d] \\
                              L_0 \ar@{-}[d]_{G} &  M \ar@{-}[dl] \\
                              K }
                   & \xymatrix@R=16pt@C=16pt{ & D \ar@{-}[d]  \\
                              D_0 \ar@{-}[d] & L \ar@{-}[dl]\ar@{-}[d] \\
                              L_0 \ar@{-}[d]_{G} &  M \ar@{-}[dl] \\
                              K  }
                   & \xymatrix@R=16pt@C=16pt{ & D \ar@{-}[d]  \\
                              D_0 \ar@{-}[d] & L \ar@{-}[dl]\ar@{-}[d]^{G} \\
                              L_0 \ar@{-}[d] &  M \ar@{-}[dl] \\
                              K }
                   & \xymatrix@R=16pt@C=16pt{ & D \ar@{-}[d]  \\
                              D_0 \ar@{-}[d] & L \ar@{-}[d]^{G} \\
                              L_0 \ar@{-}[d]_{G} &  M \ar@{-}[dl] \\
                              K  }
                  \\
                   \mbox{\Casev} & \mbox{\Casevi}
                    & \mbox{\Caseviii} & \mbox{\Casevii}
                  \end{array}
\end{equation*}

We shall say that a triple $(K,M,G)$ satisfies Condition (m) if
there are $L_0$, $L$, $D_0$ and $D$ as required in this condition.
In such case we shall also say $(L_0,L,D_0,D)$ realizes Condition
(m), omitting $L_0$ or $D_0$ if they are not needed.

\begin{rem}
Let $M/K$ be a finite extension of fields and $G$ a finite
group. One might also consider the condition
\begin{enumerate}
\setcounter{enumi}{9}
\item \label{caseix}
there is a $G$-crossed product $K$-division algebra $D_0$, for
which $D=D_0\otimes M$ is also a $G$-crossed product division
algebra.
\end{enumerate}
Here there is no explicit assumption that the maximal subfields be
related; in the spirit of previous diagrams, this condition is
described by
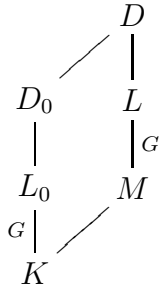
\begin{equation*}             \xymatrix@R=16pt@C=16pt{ & D \ar@{-}[dl]\ar@{-}[d]  \\
                              D_0 \ar@{-}[d] & L \ar@{-}[d]^{G} \\
                              L_0 \ar@{-}[d]_{G} &  M \ar@{-}[dl] \\
                              K}
\end{equation*}
However, \Caseix\ is equivalent to Condition \Casev. Indeed,
suppose that $(L_0,L,D_0,D)$ realizes \Caseix. Then $D$ is of index
$\card{G}$ and $D$ is also split by $L'=ML_0$. Therefore
$\dimcol{L'}{M}=\card{G}$, $L_0\cap M=K$ and hence we can take
$L'$ to be the required maximal $G$-subfield of $D$. Thus,
$(L_0,L',D_0,D)$ realizes $\Casev$. The converse implication is
obvious, taking $L = L_0 \tensor[K] M \sub D_0 \tensor[K] M = D$.
\end{rem}

\subsection{The logical implications}

The following theorem describes the relation between the nine
variants:
\begin{thm} \label{section1.1 - implications Thm}
Let $M/K$ be a finite extension of fields and $G$ a finite group.
Then the implications in Diagram \ref{section1.1 - implications
Diagram} hold, but no others.
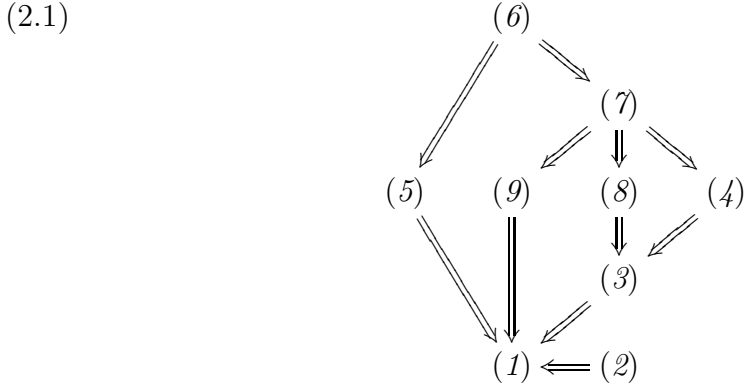
\begin{equation}\label{section1.1 - implications Diagram}
\xymatrix@R=14pt@C=18pt{
    & \Casev \ar@{=>}[rd]\ar@{=>}[ddl] & &
\\
    &  &  \Casevi \ar@{=>}[dl]\ar@{=>}[dr]\ar@{=>}[d] &
\\
    \Caseiv \ar@{=>}[ddr] & \Casevii \ar@{=>}[dd] & \Caseviii \ar@{=>}[d] & \Caseiii \ar@{=>}[dl]
\\
    & &  \Caseii \ar@{=>}[ld]. &
\\
    & \Casei & \Casenor \ar@{=>}[l] &
}
\end{equation}
%
\end{thm}%

In \Sref{sec4} we give counterexamples for the false implications,
with $G$ being a $p$-group and $K$ a number field in each case.
Let us go over the implications in diagram \ref{section1.1 -
implications Diagram}.

\begin{proof}[Proof of positive part of \Tref{section1.1 - implications Thm}]
Fix $K$, $M$ and $G$. Clearly if $(L_0,L,D_0,D)$ realizes
Condition $\Casev$, $(L_0,L,D_0,D)$ also realizes $\Casevi$ and
$(L,D_0,D)$ realizes $\Caseiv$, so that $\Casev\Rightarrow
\Caseiv,\Casevi$.

If $(L_0,L,D_0,D)$ realizes $\Casevi$ then $L_0/K$ is a
$G$-extension and hence $L=L_0M/M$ is also a $G$-extension (since
$L_0\cap M=K$). Thus, $(L_0,L,D_0,D)$ realizes $\Caseviii$. It is
clear that $L_0$ is a field of definition of $L$ (and
$\Gal(L_0/K)=G$) and hence $(L_0,L,D)$ realizes \Caseiii. As $L_0$
is a $K$-adequate $G$-extension and $L$ is an $M$-adequate
$G$-extension, $(L_0,L,D_0,D)$ realizes $\Casevii$. Therefore
$\Casevi\Rightarrow \Caseiii,\Casevii,\Caseviii$.

If $(L_0,L,D)$ realizes Condition \Caseiii\ then $\Gal(L_0/K)=G$,
$\Gal(L/M)=G$ (since $L_0\cap M=K$) and hence $(L_0,L,D_0,D)$
realizes Condition \Caseii. If $(L_0,L,D_0,D)$ realizes condition
$\Caseviii$, clearly $L_0$ is a field of definition of $L$ and
hence $(L_0,L,D)$ realizes Condition \Caseii.

Clearly when $(K,M,G)$ satisfies either of the conditions
$\Casenor, \Caseii,\Caseiv,\Casevii$, $G$ is $M$-admissible and hence
$\Casenor, \Caseii,\Caseiv,\Casevii\ \Rightarrow \Casei$.
\end{proof}

\section{Cyclic groups over number fields}\label{sec:cyclic}
For a prime $\nu$ of a number field $K$, we denote by $K_{\nu}$
the completion of $K$ with respect to $\nu$. If $L/K$ is a finite
Galois extension, $L_{\nu}$ denotes the completion of $L$ with
respect to some prime divisor of $\nu$ in $L$.

The basic criterion for admissibility over number fields is due to Schacher:
\begin{thm}[\cite{Sch}]\label{Schachers criterion}
Let $K$ be a number field and $G$ a finite group. Then $G$ is
$K$-admissible if and only if there exists a Galois $G$-extension
$L/K$ such that for every rational prime $p$ dividing $\card{G}$,
there is a pair of primes $\nu_1 ,\nu_2$ of $K$ such that each of
$\Gal(L_{\nu_i}/K_{\nu_i})$ contains a $p$-Sylow subgroup of $G$.
\end{thm}

We use this criterion in the construction of counterexamples in \Sref{sec4} and to prove the following proposition:
\begin{prop}
Let $G$ be a cyclic group. Then Conditions \Casei \ and \Caseii--\Casefinal\ are
satisfied for any extension of number fields $M/K$.
\end{prop}
\begin{proof}
It is sufficient to show that $\Casev$ is satisfied. By Chebutarev
density Theorem (applied to the Galois closure of $M/K$) there are
infinitely many primes $\nu$ of $K$ that split completely in $M$.
Let $\nu_1,\nu_2$ be two such primes that are not divisors of $2$.
By the weak version (prescribing degrees and not local extensions)
of the Grunwald-Wang Theorem (see \cite[Corollary 2]{Wan} or \cite[Chapter 10]{AT}) there is a
$G$-extension $L_0/K$ for which $\Gal((L_0)_{\nu_i}/K_{\nu_i})=G$
and thus $L_0$ is $K$-adequate, so there is a division algebra
$D_0$ containing $L_0$ as a maximal subfield, and supported by
$\{\nu_1,\nu_2\}$. As $\nu_i$ split completely in $M$ we
have $L=L_0M$ satisfies $\Gal(L_{\nu_i}/M_{\nu_i})=G$ for $i=1,2$
and hence $\Gal(L/M)=G$. Finally $D = D_0 \tensor M$ is a division
algebra by the choice of the $\nu_i$. Thus $L$ is $M$-adequate and
$(K,M,G)$ satisfies $\Casev$.
\end{proof}
\begin{exam}\label{nogal-case2.rem}
If $M/K$ is not normal, \Casenor\ does not necessarily hold for a
cyclic group $G$. Let $n\geq2$ and $M/K$ be an
extension of degree $n$ whose Galois closure $M'$ has Galois group
$\Gal(M'/K)=S_n$. Then any field $L\supseteq M$, which is Galois
over $K$, must contain $M'$ and hence there is no (adequate)
$C_n$-extension $L/M$ for which $L/K$ is Galois. In particular,
$\Casev\not\Rightarrow \Casenor$.
\end{exam}
\begin{rem}\label{easy}
If $F_1$ and $F_2$ are field extensions of $F$ such that $L = F_1
\tensor[F] F_2$ is a field, and $F_1/F$ and $L/F_1$ are Galois,
then $L$ is Galois over $F$.
\end{rem}
\begin{rem}\label{normal_case2.rem}
This shows that if $M/K$ is Galois then
$\Caseii\Rightarrow\Casenor$. In particular $\Casenor$ holds for
$G$ cyclic.
\end{rem}


We mention in this context the `linear disjointness' (LD) of
number fields, as defined and established in \cite[Prop.~2.7]{RR}:
for every finite extension $M/K$ in characteristic $0$, any
central simple algebra over $K$ contains a maximal separable
subfield $P$ that is linearly disjoint from $M$ over $K$.
This notion can be bypassed by appealing to the Chebutarev
density, as above.

\section{Tame admissibility}\label{s:tame}

The conditions of \Sref{sec:cond} can also be considered with respect to tame
$K$-admissibility. Let us recall the definition of tame admissibility.

For an extension of fields $L/K$, $\Br(L/K)$ denotes the kernel of
the restriction map $\res \co \Br(K) \ra \Br(L)$.
\forget
For number
fields we have the following isomorphism of groups, where $\Pi_K$
is the set of primes of $K$, and $(\,\cdot\,)_0$ denotes that the
sum of invariants is zero:
\begin{equation*}
\Br(L/K)\cong \(\bigoplus_{\pi\in\Pi_K} \frac{1}{\gcd_{\pi' | \pi}
\dimcol{L_{\pi'}}{K_{\pi}}}\mathbb{Z} /
\mathbb{Z}\)_0,\end{equation*}
\forgotten
%
Let $\tame{\Br(L/K)}$ be the subgroup of the relative Brauer group
$\Br(L/K)$ that consists of the Brauer classes
which are split by the maximal tame subextension of $L_{\nu}/K_{\nu}$, for every prime $\nu$ of $L$.
\forget
This is the subgroup corresponding under the above
isomorphism to
\begin{equation}
\(\bigoplus_{\pi\in\Pi_K} \frac{1}{\gcd_{\pi'|\pi} \dimcol{L_{\pi'}\cap
\tame{(K_{\pi})}}{K_{\pi}}}\Z / \Z\)_0.\end{equation}
\forgotten

Over a number field $K$, the exponent of a division algebra is
equal to its index, and so $L$ is $K$-adequate if and only if there is an
element of order $\dimcol{L}{K}$ in $\Br(L/K)$ (\cite[Proposition 2.1]{Sch}). Following this
observation one defines:
\begin{defn}
Let $K$ be a number field. We say that a finite extension $L$ of $K$ is \emph{tamely
$K$-adequate} if there is an element of order $\dimcol{L}{K}$ in
$\tame{\Br(L/K)}$.

{}Likewise, a finite group $G$ is \emph{tamely $K$-admissible} if
there is a tamely $K$-adequate $G$-extension $L/K$.
\end{defn}

\subsection{Liedahl's condition}

Let $\mu_n$ denote the set of $n$-th roots of unity in $\mC$. For
$t$ prime to $n$, let $\sigma_{t,n}$ be the automorphism of
$\mQ(\mu_n)/\mQ$ defined by $\sigma_{t,n}(\zeta)=\zeta^t$ for
$\zeta\in\mu_n$.
\begin{defn}\label{def:LC}
We say that a metacyclic $p$-group $G$ satisfies
\emph{Liedahl's condition} (first defined in \cite{Lid2}) with respect to $K$, if it has a
presentation
\begin{equation}\label{Mdef}
G = \sg{x,y \subjectto x^m = y^i,\ y^n = 1,\
x^{-1}yx = y^t}
\end{equation} such that $\s_{t,n}$ fixes $K \cap
\mQ(\mu_n)$.
\end{defn}

It follows from \cite{Lid2} (see also \cite[Corollary 2.1.7]{Nef})
that tamely $K$-admissible groups $G$ have metacyclic $p$-Sylow
subgroups that satisfy Liedahl's condition for every prime divisor
$p$ of $\card{G}$.
There are no known counterexamples
to the opposite implication.
In fact if a metacyclic $p$-group satisfies Liedahl's condition over $K$ then it is realizable over infinitely many completion of $K$ (see \cite{Lid2}).

\begin{rem}\label{section0 - remark on down-inheritance of Lid cond}
Note that if a metacyclic $p$-group $G$ satisfies Liedahl's
condition over $M$, then it satisfies the condition over every
subfield $K$.
\end{rem}


The following is shown in \cite[Theorem~30]{Lid2} for $G$ a $p$-group, and in \cite[Theorem~2.3.1]{Nef} for $G$ solvable.
\begin{thm}\label{section2- remark about lifting of division algs}
Let $K$ be a number field and $G$ a solvable group whose Sylow
subgroups satisfy Liedahl's condition. Then $G$ is tamely
$K$-admissible.
\end{thm}
\begin{rem}\label{tame.rem}In fact the proof of \cite[Theorem~2.3.1]{Nef}
shows that there is a $G$-extension $L_0/K$ and $D_0\in
\Br(L_0/K)_{tr}$ such that $D:=D_0\otimes_\mQ K$ remains a
division algebra.

In particular $L:=L_0\otimes_K M$ is an $M$-adequate field which is a $G$-extension of $M$.
Thus, not only $G$ is $M$-admissible but there is also a
$G$-crossed product division algebra $D$ and a maximal subfield
$L$ so that both are defined compatibly over $\mQ$. \end{rem}


As a corollary one has (see \cite{Nef}):
\begin{cor}\label{section2-sylow meta cyclic admissibility corollary}
Let $K$ be a number field. Let $G$ be a solvable group such that
the rational prime divisors of $\card{G}$ do not decompose (i.e.
have a unique prime divisor) in $K$. Then $G$ is $K$-admissible if
and only if its Sylow subgroups are metacyclic and satisfy
Liedhal's condition.
\end{cor}

\subsection{Fields of definition for tame admissibility}

The conditions of \Sref{sec:cond} can also be considered with respect to tame
$K$-admissibility. Let $G$ be a solvable group and $K,M$ number
fields. By \Pref{section2- remark about lifting of division
algs}, if $G$ is tamely $M$-admissible then there is a tamely
$K$-adequate $G$-extension $L_0/K$ for which $L=L_0M$ is
$M$-adequate (and hence tamely $M$-adequate). For $m = 1,
\dots,\final$, let $(m^*)$ denote the condition $(m)$, where every
adequate extension is assumed to be tamely adequate, and an
admissible group is assumed tamely admissible. More precisely for
$m = \ref{caseiv}, \ref{casev}$ we consider
\begin{enumerate}
\item[$(\ref{caseiv}^*)$]
there is a $K$-division algebra $D_0$ and a $G$-extension
$L/M$ for which $[D]=[D_0\otimes M] \in \tame{\Br(L/M)}$ and $L$
is a maximal subfield of $D$,
\end{enumerate}
and
\begin{enumerate}
\item[$(\ref{casev}^*)$]
there is a $K$-division algebra $D_0$ and a maximal subfield $L_0$
which is a $G$-extension of $K$ so that $L_0\cap M=K$,
$D_0\in \tame{\Br(L_0/K)}$ and $L=L_0M$ is a maximal subfield of
$D=D_0\otimes M$ (and hence $[D]\in \tame{\Br(L/M)}$).
\end{enumerate}
\begin{cor}\label{section1.1 - Corollary on tame admissibility}
Let $G$ be a solvable group and $M/K$ a finite extension of number
fields. Then the conditions $(1^*)$ and $(3^*)$--$(\final^*)$ are
all equivalent.
\end{cor}
\begin{proof}
With the added conditions the implications given in \eq{section1.1
- implications Diagram} clearly continue to hold. But by \Rref{tame.rem}
the implication $(\ref{casei}^*) \Rightarrow (\ref{casev}^*)$ also
holds.

\end{proof}


\section{Examples}\label{sec4}

In this section we give counterexamples for all the implications
not claimed in \Tref{section1.1 - implications Thm}. In all the
examples, the group $G$ is a $p$-group. This shows that
Diagram~\ref{section1.1 - implications Diagram} describes all the
correct implications even for $p$-groups.

Let us first show that none of the conditions $\Casenor$,
$\Caseiv$ or $\Casevii$ imply any other condition except \Casei.
For this, by the implication Diagram \ref{section1.1 -
implications Diagram}, it is sufficient to show that $\Casenor
\not\Rightarrow \Casevii$, $\Casenor \not\Rightarrow \Caseiv$, $
\Casenor \not\Rightarrow \Caseii$, $\Caseiv \not\Rightarrow
\Casevii$, $\Casevii \not\Rightarrow \Caseiv$, $\Caseiv
\not\Rightarrow \Caseii$ and that $\Casevii \not\Rightarrow
\Caseii$. We will show that $\Casevii \not \Rightarrow \Caseiv$ by
demonstrating that $\Casevi\not\Rightarrow \Caseiv$. In fact an
example for $\Casevi\not\Rightarrow \Caseiv$ will show that no
other condition, except \Casev, implies Condition \Caseiv. To
complete the proof we should also prove $\Casenor\not\Rightarrow
\Casev$, $\Caseviii\not\Rightarrow \Casevii$,
$\Caseviii\not\Rightarrow \Caseiii$, $\Caseiii \not\Rightarrow
\Casevii$ and $\Caseiii\not\Rightarrow \Caseviii$.

\begin{rem}
Note that $\Casev\not\Rightarrow \Casenor$ follows from
\Rref{nogal-case2.rem}
\end{rem}
\begin{exam}[$\Casenor\not\Rightarrow \Caseii,\Caseiv\not\Rightarrow \Caseii,\Casevii\not\Rightarrow \Caseii$]
Let $p\equiv 1 \pmod{4}$, $G=(\Z/p\Z)^3$ and $K=\mQ(i,\sqrt{p})$.
Note that $p$ splits in $K$. Denote the prime divisors of $p$ in
$K$ by $\nu_1,\nu_2$.

Let $\overline{K_{\nu_i}(p)}^{\ab}$ be the maximal abelian pro-$p$
extension of $K_{\nu_i}$. By local class field theory the Galois
group $\Gal(\overline{K_{\nu_i}(p)}^{\ab}/K_{\nu_i})$ is
isomorphic to the pro-$p$ completion of the group $K_{\nu_i}^*$
which is $\Z_p^n$ where $n=\dimcol{K_{\nu_i}}{\mQ_p}+1=3$ (see
\cite{Ser3}, Chapter 14, Section 6).

Since $K_{\nu_1}=K_{\nu_2}=\mQ_p(\sqrt{p})$ this shows $G$ is
realizable over $K_{\nu_1},K_{\nu_2}$. 
By the Grunwald-Wang Theorem there a $(\Z/p^2\Z)^3$-extension $\hat{M}/K$ 
such that
$\hat{M}_{\nu_i}$ is the maximal abelian extension of exponent $p^2$ of
$K_{\nu_i}$, namely the unique $(\Z/p^2\Z)^3$-extension of
$K_{\nu_i}$. Let $M=\hat{M}^G$, so that $\Gal(M/K)\cong G$.

Since $\hat{M}/M$ and $M/K$ both have full local degrees at $v_1,v_2$,
both are adequate $G$-extensions. Note that $\hat{M}$ is also Galois
over $K$.
By choosing $L=\hat{M}$, we deduce that $(K,M,G)$ satisfies
conditions \Casenor \ and \Casevii. To show that $(K,M,G)$
satisfies \Caseiv\ it suffices to notice that $\nu_1,\nu_2$ have
unique prime divisors $\varpi_1,\varpi_2$ in $M$. Every division
algebra $D$ whose invariants are supported in
$\{\varpi_1,\varpi_2\}$ is $K$-uniformly distributed and hence
$D\in \Im(\res^M_K)$. Take $D$ with
\begin{equation*}
\inv_{\varpi_1}(D)=\frac{1}{p^3}, \quad \inv_{\varpi_2}(D)=-\frac{1}{p^3}
\end{equation*} and $\inv_w(D)=0$ for any other prime $\varpi$ of $M$.
We then have $D\in \Im(\res^M_K)$, $D$ is a $G$-crossed product
division algebra and hence $(K,M,G)$ satisfies \Caseiv.

Let us show \Caseii\ is not satisfied. Suppose on the contrary
that there is a triple $(L_0,L,D)$ realizing \Caseii. By
\Rref{easy}, $L/K$ is Galois and
\begin{equation*}
\Gal(L/K)\cong \Gal(L/L_0)\ltimes \Gal(L/M)\cong G\ltimes_\phi G
\end{equation*}
via some homomorphism $\phi:G\ra \Aut(G)=\GL[3](\mathbb{F}_p)$. As
$G$ is a $p$-group, $\phi$ is a homomorphism into some $p$-Sylow
subgroup $P$ of $\GL[3](\mathbb{F}_p)$. These are all conjugate,
so we can choose a basis $\{v_1,v_2,v_3\}$ of $\F_p^3$ for which
$P$ is the Heisenberg group (in other words the unipotent radical
of the standard Borel subgroup), generated by the transformations:
\begin{equation*}
\phi_x(a,b,c)=(a+b,b,c),\ \phi_y(a,b,c)=(a,b+c,c), \
\phi_u(a,b,c)=(a+c,b,c)
\end{equation*} which correspond to the
matrices
\begin{equation*}x=\left( \begin{array}{ccc}
1 & 1 & 0 \\
0 & 1 & 0 \\
0 & 0 & 1 \end{array} \right)
, y=\left( \begin{array}{ccc}
1 & 0 & 0 \\
0 & 1 & 1 \\
0 & 0 & 1 \end{array} \right),
u=\left( \begin{array}{ccc}
1 & 0 & 1 \\
0 & 1 & 0 \\
0 & 0 & 1 \end{array} \right).\end{equation*} Note that $P$ has
the presentation %
\begin{equation*} %
P=\langle x,y,u \subjectto x^p=y^p=u^p=[x,u]=[y,u]=1,[y,x]=u\rangle.%
\end{equation*}

Every subgroup of the form $\F_p^2\ltimes G$ is a maximal subgroup
of $G\ltimes G$ and thus the Frattini subgroup $\Phi$ of
$G\ltimes_\phi G$ is contained in $1\ltimes G$. Now the subgroup
$H=\langle v_1,v_2 \rangle\leq G$ is invariant under the action of
$P$ and hence under the action of $G$ via
$\phi$. So, $G\ltimes_\phi H\leq G\ltimes_\phi G$ is a
maximal subgroup and $\Phi\leq 1\ltimes H$. This shows that
$\dim_{\mathbb{F}_p}G/\Phi\geq 4$ and thus $G\ltimes_\phi G$ is not
generated by less than $4$ elements. Therefore $G\ltimes G$ is not
realizable over $\mathbb{Q}_p(\sqrt{p})$.

On the other hand both $L/M$ and $M/K$ have full rank at
$\varpi_i$ and $\nu_i$ and hence
$\Gal(L_{\varpi_i}/K_{\nu_i})=G\ltimes G$ which is a contradiction
as $G\ltimes G$ is not realizable over $K_{\nu_i}$. Thus,
$(K,M,G)$ does not satisfy Condition \Caseii.
\end{exam}

\begin{exam}[$\Casenor\not\Rightarrow \Casevii, \Caseviii\not\Rightarrow \Casevii,\Caseviii\not\Rightarrow \Caseiii$]
Let $p\equiv 1 \pmod{4}$, $K=\mQ(i)$ and $\nu_1,\nu_2$
the two prime divisors of $p$ in $K$. Let $G=\F_p^p$ and
$P=\F_p\wr (\Z/p\Z)$ so that $P=G \rtimes \sg{x}$ where $x^p = 1$.

The maximal $p$-extension $\oline{\mQ_p(p)}$ has Galois group
$\overline{G_{\mQ_p}(p)} := \Gal(\oline{\mQ_p(p)}/\mQ_p)$ which is
a free pro-$p$ group on two generators. As $P$ is generated by two
elements it is realizable over $\mQ_p$. Since $P$ is a wreath
product of abelian groups it has a generic extension over $K$ and
hence by \cite{Sal1} there is a $P$-extension $L/K$ for which
$\Gal(L_{\nu_i}/K_{\nu_i})=P$ for $i=1,2$. Let us choose $M=L^G$
the $G$-fixed subfield of $L$.

Then clearly $L/M$ is an $M$-adequate extension which is defined
over $K$ since \begin{equation*} \Gal(L/K)\cong \Gal(M/K)\ltimes
\Gal(L/M). \end{equation*} The subfield $L_0=L^{\sg{x}}$ is
$K$-adequate since $\dimcol{(L_0)_{\nu_i}}{K_{\nu_i}}=p^p$ for $i=1,2$ and
hence $(K,M,G)$ satisfies Condition \Caseviii.

(We write $(L_0)_{\nu_i}$ even though $L_0/K$ is not Galois, since
$\nu_i$ has a unique prime divisor in $L_0$ for $i=1,2$.)

Now since $G$ is an abelian group of rank $p>2$, $G$ is not
realizable over $K_{\nu_1},K_{\nu_2}\cong \mQ_p$ and hence not
$K$-admissible. It follows that $(K,M,G)$ does not satisfy
Condition \Casevii. In order for $(K,M,G)$ to satisfy Condition
$\Caseiii$ there should be a $G$-extension $L_0/K$ for which
$L_0M$ is $M$-adequate. In particular,
$\Gal((L_0M)_{\nu_1}/M_{\nu_1})\cong G$ and hence
$\Gal((L_0)_{\nu_1}/K_{\nu_1})\cong G$ which contradicts the fact
that $G$ is not realizable over $K_{\nu_1}\cong \mQ_p$. Thus
$(K,M,G)$ does not satisfy Condition \Caseiii\ either.

By \Rref{normal_case2.rem}, as $M/K$ is Galois, ${\Caseviii}
\Rightarrow \Caseii \Rightarrow\Casenor$ and hence $(K,M,G)$ also
satisfies \Casenor.
\end{exam}

\begin{exam}[$\Caseiii\not\Rightarrow \Casevii$,
$\Caseiii\not\Rightarrow \Caseviii$] Let $p\equiv 1 \pmod{4}$ and
$\nu$ be its unique prime divisor in $K=\mQ(\sqrt{p})$. Let
$M=\mQ(\sqrt{p},i)$ and $G=(\Z/p\Z)^3$.

By the Grunwald-Wang Theorem, there is a Galois $G$-extension
$L_0/K$ for which $\Gal((L_0)_\nu/K_\nu)=G$. Thus $L=L_0M$ is a
Galois $G$-extension of $M$ such that
$\Gal(L_{\nu_i}/M_{\nu_i})=G$ for each of the two prime divisors
$\nu_1,\nu_2$ of $\nu$ in $M$. It follows that $L$ is $M$-adequate
and $(K,M,G)$ satisfies Condition $\Caseiii$. But as $p$ has a
unique prime divisor in $K$ and $G$ is not metacyclic, $G$ is not
$K$-admissible and hence $(K,M,G)$ does not satisfy Condition
$\Casevii$.

Let us also show that $(K,M,G)$ does not satisfy Condition
$\Caseviii$. Assume, on the contrary, that $(L_0,L,D_0,D)$
realizes $\Caseviii$. Then, as $L_0$ is $K$-adequate there are two
primes $\varpi_1,\varpi_2$ of $K$ for which
$\dimcol{(L_0)_{\varpi_i}}{K_{\varpi_i}}=\card{G}$. Without loss
of generality we assume $\varpi_1\not=v$ (otherwise take
$\varpi_2$). Then $\Gal(L_{\varpi_1}/M_{\varpi_1})\cong G$ since
$(L_0)_{\varpi_1}\cap M_{\varpi_1}=K_{\varpi_1}$. This is a
contradiction since tamely ramified extensions (such as
$L_{\varpi_1}/M_{\varpi_1}$) have metacyclic Galois groups. Thus
$(K,M,G)$ does not satisfy Condition $\Caseviii$.
\begin{rem}
Let us also show that $(K,M,G)$ does not satisfy $\Caseiv$ (so
that this example will also show that $\Caseiii\not\Rightarrow
\Caseiv$). Assume on the contrary that there is a $(L,D_0,D)$ that realizes $\Caseiv$. Since $D$ contains $L$ as a maximal
subfield, $\Gal(L_{\nu_i}/M_{\nu_i})=G$ and
$\inv_{\nu_i}(D)=\frac{m_i}{p^3}$ where $(m_i,p)=1$, for $i=1,2$.
Note that $G$ is realizable over $M_\nu$ only for divisors $\nu$
of $p$, so that $\inv_u(D)=\frac{m_u}{p^2}$ for suitable $m_u \in
\Z$ for any $u\not=\nu_1,\nu_2$. Now, since $D$ is in the image of
the restriction, we have $m_1=m_2$. The sum of $M$-invariants of
$D$ is an integer and hence $p \divides m_1+m_2=2m_1$ which
contradicts $(m_i,p)=1$.
\end{rem}

\end{exam}

\begin{exam}[$\Caseiv\not\Rightarrow \Casevii$]
Let $p$ be any odd prime, and $q$ a prime $\equiv 1 \pmod{p}$. Let
$K=\mQ(\sqrt{p})$, so that $q$ splits (completely) in $K$. Let
$\nu$ be the prime divisor of $p$ in $K$ and $\varpi$ a prime
divisor of $q$ in $K$. Let $M$ be a $\Z/p\Z$-extension of $K$ in
which $\nu$ splits completely and $\varpi$ is inert. Let
$G=(\Z/p\Z)^3$.

Consider the $K$-division algebra $D_0$ whose invariants are:
\begin{equation*}
\inv_\nu(D_0)=\frac{1}{p^3},\quad
\inv_w(D_0)=-\frac{1}{p^3}\end{equation*} and $\inv_u(D_0)=0$ for
any other prime $u$ of $K$. Now $D=D_0\otimes_K M$ has
$M$-invariants $\inv_{\nu_i}(D)=\frac{1}{p^3}$ for
the prime divisors $\nu_1,\nu_2,\ldots,\nu_p$ of $\nu$ in $M$,
$\inv_{\varpi'}(D)=-\frac{1}{p^2}$ for the prime divisor $\varpi'$ of $\varpi$ and
$\inv_u(D)=0$ for any other prime $u$ of $M$. Note that $G$ is
realizable over $M_{\nu_i}\cong K_\nu$ and since $q\equiv 1 \pmod{p}$, $(\Z/p\Z)^2$ is
realizable over $M_{\varpi'}$. By the Grunwald-Wang Theorem, there is a Galois
$G$-extension $L/M$ for which:
\begin{center}
$\Gal(L_{\nu_i}/M_{\nu_i})=G$\ for\ $i=1,\ldots,p$,\ and\ $\Gal(L_{\varpi'}/M_{\varpi'})=(\Z/p\Z)^2.$
\end{center}

Thus $L$ is a maximal subfield of $D$ and $(K,M,G)$ satisfies
condition $\Caseiv$. Since $p$ has a unique prime divisor in $K$ and $G$ is not
metacyclic we deduce $G$ is not $K$-admissible and hence $(K,M,G)$ does not satisfy Condition
$\Casevii$.
\end{exam}

\begin{exam}[$\Casenor\not\Rightarrow \Caseiv, \Casevi\not\Rightarrow \Caseiv$]\label{E5.7}
Let $p\geq 13$ be a prime such that $p\equiv 1 \pmod{4}$. Let
$K=\mQ(\mu_p)$ and $M=\mQ(\mu_{4p^2})=\mQ(i,\mu_{p^2})$. Let $G$
be the following metacyclic group of order $p^3$:
\begin{equation}\label{thisG}
G=\sg{x,y \subjectto x^{p}=y^{p^2}=1,\, x^{-1}yx=y^{p+1}}.
\end{equation}
Note that $p$ splits in $\mQ(i)$ and has exactly two prime
divisors $\nu_1,\nu_2$ in $M$. Let $u$ be the unique prime divisor
of $p$ in $K$.
%

Let us first show that $(K,M,G)$ does not satisfy Condition
$\Caseiv$. As $M$ does not satisfy Liedahl's condition, $G$ is not
realizable over $M_\nu$ for any $\nu\not=\nu_1,\nu_2$. Assume on
the contrary there is an $M$-adequate $G$-extension $L/M$ and an
$M$-division algebra $D$ which is defined over $K$ and has a
maximal subfield $L$. Then necessarily:
$\inv_{\nu_1}(D)=\inv_{\nu_2}(D)=\frac{a}{p^3}$ for some
$(a,p)=1$. But as the sum of invariants of $D$ is $0$ and $G$ is
not realizable over any other $\nu$ we have $p \divides 2a$ or $p
\divides a$. Contradiction.

To prove that $(K,M,G)$ satisfies Condition \Casevi\, we shall need the following lemma:

\begin{lem}\label{4.8}
Let $p\geq 11$ be a prime, $k=\mQ_p(\mu_p)$ and $G$ the group
defined in \emph{\eq{thisG}}. Then, given a $G$-extension $m/k$,
there is a $G$-extension $l/k$ for which $m\cap l=k$.
\end{lem}
\begin{proof}
For any $G$-extension $l/k$ we note that $\Gal(l\cap m/k)$ is an
epimorphic image of $G$ and as such it is either $G$ or an abelian
group. Thus if $l$ intersects with $m$ non-trivially then it also
intersects with $m'=m^{\langle y^p\rangle}$ (the fixed field of
$y^p$ which also corresponds to the abelianization of $G$). We
note that $\Gal(m'/k)=(\Z/p\Z)\times (\Z/p\Z)$. The maximal
abelian group realizable over $k$ is of rank $p-1$, and since
$\frac{p-1}{2} \geq 4$ there is a $(\Z/p\Z)^2$-extension $l'/k$
which is disjoint from $m'$ and for which the epimorphism
$\pi:G_k\ra \Gal(l'/k)$ splits through a free pro-$p$ group of
rank $\frac{p-1}{2}$. Thus $l'$ is disjoint from $m'$ and hence to
$m$. Embedding $l'$ into a $G$-extension produces a $G$-extension
which is disjoint to $m$. This is possible since the following
embedding problem for $G_k$:
\begin{equation*}
\xymatrix{    & {G}_{k} \ar[d] \ar@{-->}[ddl]& \\
                                & F_p(\frac{p-1}{2}) \ar[d] \ar@{-->}[dl] \\
                              G \ar[r]& (\Z/p\Z)^2 \ar[r] & 0, } \end{equation*}
splits through a free pro-$p$ group of large enough rank and hence has a surjective solution.
\end{proof}

Let us prove Condition \Casevi \, is satisfied. Let
$\sigma_{p+1}\in \Gal(\mQ(\mu_{p^2})/\mQ)$ be the automorphism
that sends $\sigma_{p+1}(\zeta)=\zeta^{p+1}$ where $\zeta$ is a
primitive root of unity of order $p^2$. Thus $\sigma_{p+1}$ fixes
$\mu_p$ and hence $\sigma_{p+1}\in \Gal(\mQ(\mu_{p^2})/K)$. As $G$
satisfies Liedahl's condition over $K$, $G$ is realizable over
infinitely many primes of $K$ (see the proof of \cite[Theorem
29]{Lid2} or \cite[Theorem 2.3.1]{Nef}), so choose one such prime
$\varpi$ which is not a divisor of $p$. Since
$\dimcol{K_u}{\mQ_p}=p-1\geq 11$, it follows from \Lref{4.8} that
$G$ is also realizable over $K_u$ and furthermore
 there is a $G$-extension $L_0^p/K_u$ for which
$M_u\cap L_0^p=K_u$. 

%


By Theorems $6.4(b)$ and  $2.5$ of \cite{Neu2} (see also  \cite[Proposition 1.2.13]{Nef}), there is a $G$-extension $L_0/K$ for
which $\Gal((L_0)_w/K_w)=G$ and $(L_0)_u=L_0^p$. Hence $L_0$ is
$K$-adequate. Let $L=L_0M$. As $M_u\cap L_0^p=K_u$ we have
$\Gal(L_{\nu_i}/M_{\nu_i})=G$ for $i=1,2$. Thus $L/M$ is an
$M$-adequate $G$-extension and $(K,M,G)$ satisfies Condition
\Casevi.

By \Rref{normal_case2.rem}, as $M/K$ is
Galois, $\Casevi\Rightarrow\Casenor$. Thus,  $(K,M,G)$ also
satisfies \Casenor. This concludes the proof of \Eref{E5.7}.
\end{exam}


\newcommand\arxiv[1]{{\texttt{#1}}}


\begin{thebibliography}{9}%


\bibitem{AT}{\sc E. Artin, J. Tate}, Class field theory. Advanced Book Classics. Addison-Wesley Publishing Company, Advanced Book Program, Redwood City, CA, 1990.


\bibitem{Lab} {\sc J.~P.~Labute},
Classification of Demushkin groups. {\sl Canad. J. Math.} {\bf 19}
(1967), 106--132.

\bibitem{Lid2} {\sc S.~Liedahl},
{\it Presentations of metacylic $p$-groups with applications to
$K$-admissibility questions}, {\sl J. Algebra} {\bf 169}(3)
(1994), 965--983.


\bibitem{Nef} {\sc D.~Neftin},
{\it Admissibility of finite groups over number fields}, Ph.D.
Thesis, Technion 2011.


\bibitem{PaperI}{\sc D.~Neftin and U.~Vishne},
{\it Realizability and admissibility under extension of $p$-adic
and number fields}, submitted.


\bibitem{Neu} {\sc J.~Neukirch},
{\it On solvable number fields}, {\sl Invent. Math.} {\bf 53} (1979),
no. 2, 135--164.

\bibitem{Neu2} {\sc J.~Neukirch}, Uber das Einbettungsproblem der algebraischen Zahlentheorie.  {\sl Invent. Math.} {\bf 21} (1973), 59--116.



\bibitem{RR}{\sc A.S.~Rapinchuk and I.R.~Rapinchuk}, On division algebras having the samemaximal
subfields, {\sl Manuscripta Math.} {\bf 132}, 273--293, (2010).

\bibitem{Sal1} {\sc D.~Saltman},
{\it Generic Galois extensions}, {\sl Proc. Nat. Acad. Sci. U.S.A.
77} (1980), {\bf 3}, part 1, 1250--1251.


\bibitem{Sch} {\sc M.~Schacher}, Subfields of division rings. I. {\sl J. Algebra} {\bf 9} (1968) 451--477.

\bibitem{Serre} {\sc J.-P.~Serre}, Galois Cohomology, Springer, 1964 (English trans.~1996).


\bibitem{Ser3} {\sc J.-P. Serre}, Local fields. Graduate Texts in Mathematics, {\bf 67}. Springer-Verlag, New York-Berlin, 1979.


\bibitem{Wan} {\sc S.~Wang}, {\it On Grunwald's theorem}, {\sl Ann. of Math. (2)} {\bf 51} (1950), 471--484.







\end{thebibliography}
\end{document}